\documentclass{amsart}
\usepackage{amssymb,amsmath}
\usepackage{verbatim}

\theoremstyle{plain}
\newtheorem{theorem}{Theorem}[section]
\newtheorem{lemma}[theorem]{Lemma}
\newtheorem{corollary}[theorem]{Corollary}
\newtheorem{proposition}[theorem]{Proposition}
\newtheorem{cor}[theorem]{Corollary}

\theoremstyle{definition}

\newtheorem{defi}[theorem]{Definition}

\numberwithin{equation}{section}

\newcommand{\C}{\mathbb{C}}
\newcommand{\T}{\mathcal{T}}
\newcommand{\I}{\mathcal{I}}
\newcommand{\N}{\mathbb{N}}
\newcommand{\E}{\mathcal{E}}
\renewcommand{\O}{\mathcal{O}}

\begin{document}
\title[Endomorphisms and Toeplitz Algebras]{Toeplitz Algebras of Correspondences and Endomorphisms of Sums of Type I Factors}
\author{Philip M. Gipson}
\address{Department of Mathematics \\ State University of New York College at Cortland \\ Cortland, NY 13045-0900}
\email{philip.gipson@cortland.edu}

\begin{abstract}
It is a well-known fact that endomorphisms of $B(H)$ are intimately connected with families of mutually orthogonal isometries, i.e.\ with representations of the so-called Toeplitz $C^*$-algebras. In this paper we consider a natural generalization of this connection between the representation theory of certain $C^*$-algebras associated to graphs and endomorphisms of certain von Neumann subalgebras of $B(H)$. Our primary results give criteria by which it may be determined if two representations give rise to equal or conjugate endomorphisms.
\end{abstract}

\maketitle

\section*{INTRODUCTION}

The connection between endomorphisms of factors and families of isometric operators has been explored most notably by Arveson \cite{arveson}, among others \cite{brenken1,brenken2,laca,longo}. In \cite{laca}, Laca determines that given a normal $*$-endomorphism $\alpha$ of $B(H)$ there exists an $n\in\N\cup\{\infty\}$ and $*$-representation $\pi:\E_n\to B(H)$, where $\E_n$ denotes the Toeplitz algebra for $n$ orthogonal isometries $v_1,...,v_n$, such that
\begin{equation*}\alpha(T)=\sum_{i=1}^n\pi(v_i)T\pi(v_i)^*\end{equation*}
for each $T\in B(H)$. The $n$ value is unique but the representation $\pi$ may differ by automorphisms of $\E_n$ which arise specifically from unitary transformations of the Hilbert space $\ell^2(\{v_1,...,v_n\})\subseteq\E_n$ \cite[Proposition 2.2]{laca}. 

In \cite{brenken1} and \cite{brenken2} Brenken extends the connection by determining that, for a von Neumann algebra which can be decomposed into a direct sum of Type I factors, a certain class of $*$-endomorphisms correspond to representations of certain $C^*$-algebras associated with (possibly infinite) matrices which arise as the adjacency matrices for directed graphs. The $*$-endomorphisms studied by Brenken are, however, required to either be unital \cite{brenken1} or satisfy a number of restrictive conditions \cite[pp 25]{brenken2}.

Our results extend the work of Brenken by eliminating the conditions imposed on the $*$-endomorphisms. The fundamental difference in our approach is that we will investigate representations of the Toeplitz algebra of a $C^*$-correspondence, whereas Brenken concerned himself with the so-called relative Cuntz-Pimnser algebras. We recover Brenken's results as a special case when the endomorphisms are assumed unital. Our primary result in this line is Theorem \ref{theorem3}.

Our results also continue the spirit of Laca's investigations which allow for equivalencies between $*$-endomorphisms to be encoded as transformations of an underlying linear object: a Hilbert space in the case of Laca and a $C^*$-correspondence in the present work. Our primary results along these lines are Theorems \ref{theorem1} and \ref{theorem2}.

\section{PRELIMINARIES}
First we will establish our terminology and notation.

\begin{defi} A \emph{graph} is a tuple $E=(E^0,E^1,r,s)$ consisting of a \emph{vertex set} $E^0$, an \emph{edge set} $E^1$, and \emph{range} and \emph{source} maps $r,s:E^1\to E^0$.
\end{defi}

We will only consider graphs where $E^0$ and $E^1$ are at most countable.

\begin{defi} Let $A$ be a $C^*$-algebra. A set $X$ is a \emph{$C^*$-correspondence over $A$} provided that it is a right Hilbert $A$-module and there is a $*$-homomorphism $\phi:A\to L(X)$, where $L(X)$ denotes the space of adjointable $A$-module homomorphisms from $X$ to itself. 
\end{defi} 

Given a $C^*$-correspondence $X$ over $A$, we will denote the $A$-valued inner product of $x,y\in X$ by ``$\langle x,y\rangle_A$" (perhaps omitting the $A$) and the right action of $a\in A$ on $x\in X$ will be written as ``$x\cdot a$". The map $\phi:A\to L(X)$ may sometimes be written as $\phi_X$ for clarity.

Our primary objects of study will be certain $C^*$-correspondences which arise from graphs. The following constructions are due originally to Fowler and Raeburn \cite[Example 1.2]{fowlerraeburn}, although we will adopt the modern convention for the roles of $r$ and $s$. 

\begin{defi} Given a graph $E$, the \emph{graph correspondence} $X(E)$ is the set of all functions $x:E^1\to\C$ for which $\hat{x}(v):=\sum_{e\in s^{-1}(v)}|x(e)|^2$ extends to a function $\hat{x}\in C_0(E^0)$. We give $X(E)$ the structure of a $C^*$-correspondence over $C_0(E^0)$ as follows:
\begin{align*}x\cdot a&:e\mapsto x(e)a(s(e)),\\
\phi(a)x&:e\mapsto a(r(e))x(e),\\
\langle x,y\rangle&:v\mapsto\sum_{e\in s^{-1}(v)}\overline{x(e)}y(e).\end{align*}
which is to say that $a\in C_0(E^0)$ acts on the right of $X(E)$ as multiplication by $a\circ s$ and acts on the left as multiplication by $a\circ r$.
\end{defi}

The sets $\{\delta_e:e\in E^1\}$ and $\{\delta_v:v\in E^0\}$ are dense in $X(E)$ and $C_0(E^0)$, respectively, in the appropriate senses. For $e\in E^1$ and $v\in E^0$ we have the following useful relations: $\langle \delta_e,\delta_e\rangle=\delta_{s(e)}$, $\delta_e\cdot\delta_v=\delta_e$ if $v=s(e)$ and is $0$ otherwise, and $\phi(\delta_v)\delta_e=\delta_e$ if $v=r(e)$ and is $0$ otherwise.

\begin{defi}\label{toeplitzrepdef} Given a $C^*$-correspondence $X$ over $A$ and given another $C^*$-algebra $B$, a \emph{Toeplitz representation} of $X$ in $B$ is a pair $(\sigma,\pi)$ consisting of a linear map $\sigma:X\to B$ and a $*$-homomorphism $\pi:A\to B$ such that for all $x,y\in X$ and $a\in A$
	\begin{enumerate}
	\item $\sigma(x\cdot a)=\sigma(x)\pi(a)$,
	\item $\sigma(\phi(a)x)=\pi(a)\sigma(x)$, and
	\item $\pi(\langle x,y\rangle)=\sigma(x)^*\sigma(y).$
	\end{enumerate}
\end{defi}

For a graph correspondence $X(E)$, a Toeplitz representation $(\sigma,\pi)$ is determined entirely by the values $\{\sigma(\delta_e):e\in E^1\}$ and $\{\pi(\delta_v):v\in E^0\}$. Property (iii) of a Toeplitz representation guarantees that $\sigma(\delta_e)$ is a partial isometry with source projection $\pi(\delta_{s(e)})$.

\begin{defi}\cite[Proposition 1.3]{fowlerraeburn} Given a $C^*$-correspondence $X$ over $A$, the \emph{Toeplitz algebra of $X$}, denoted $\T_X$, is the $C^*$-algebra which is universal in the following sense: there exists a Toeplitz representation $(\sigma_u,\pi_u)$ of $X$ in $\T_X$ such that if $(\sigma,\pi)$ is another Toeplitz representation of $X$ in a $C^*$-algebra $B$ then there exists a unique $*$-homomorphism $\rho_{\sigma,\pi}:\T_X\to B$ such that $\sigma=\rho_{\sigma,\pi}\circ\sigma_u$ and $\pi=\rho_{\sigma,\pi}\circ\pi_u$.
\end{defi}

That $\T_X$ exists was proven by Pimnser in \cite{pimsner}.

Given a graph $E$ we may consider the Toeplitz algebra of its graph correspondence, cumbersomely denoted $\T_{X(E)}$. Unless there is danger of confusion, we will abuse notation and make no distinction between elements of $X(E)$ and $C_0(E^0)$ and their images in $\T_{X(E)}$ under the universal maps $\sigma_u$ and $\pi_u$. 

If $\tau:\T_{X(E)}\to B(H)$ is a $*$-representation then, for each $e\in E^1$, $\tau(\delta_e)$ is a partial isometry with source projection $\tau(\delta_{s(e)})$ and range projection contained in $\tau(\delta_{r(e)})$.

If $E$ is the graph with but a single vertex and $n$ edges then $X(E)$ is a Hilbert space of dimension $n$ and $\T_{X(E)}$ is isomorphic to the classical Toeplitz algebra $\E_n$. In this case the elements $\{\delta_e:e\in E^1\}$ are precisely the generating isometries of $\E_n$. The space $X(E)$ plays a significant role in the analysis of endomorphisms of $B(H)$ in \cite{laca}, and it is for this reason that we are considering the generalized Toeplitz algebras $\T_{X(E)}$ in our investigations.

\section{COHERENT UNITARY EQUIVALENCE}

Two graphs $E$ and $F$ are isomorphic if there are two bijections $\psi^0:E^0\to F^0$ and $\psi^1:E^1\to F^1$ for which $r_F\circ\psi^1=\psi^0\circ r_E$ and $s_F\circ\psi^1=\psi^0\circ s_E$. In order to encode such an isomorphism at the level of the graph correspondences $X(E)$ and $X(F)$, we offer the following novel definition.

\begin{defi} Let $X$ and $Y$ be $C^*$-correspondences over $A$ and $B$, respectively. A \emph{coherent unitary equivalence} between $X$ and $Y$ is a pair $(U,\alpha)$ consisting of a bijective linear map $U:X\to Y$ and a $*$-isomorphism $\alpha:A\to B$ for which 
	\begin{enumerate}
	\item $U(x\cdot a)=(Ux)\cdot \alpha(a)$ for all $x\in X$ and $a\in A$,
	\item $U(\phi_X(a)x)=\phi_Y(\alpha(a))Ux$ for all $x\in X$ and $a\in A$, and
	\item $\langle Ux,y\rangle_Y=\alpha(\langle x,U^{-1}y\rangle_X)$ for all $x\in X$ and $y\in Y$.
	\end{enumerate}
Routine calculations will verify that coherent unitary equivalence is an equivalence relation.
\end{defi}

\begin{proposition} If $E$ and $F$ are isomorphic graphs then $X(E)$ and $X(F)$ are coherently unitarily equivalent.
\end{proposition} 
\begin{proof} We'll assume $(\psi^0,\psi^1)$ to be an isomorphism from $F$ to $E$.

For $a\in C_0(E^0)$, $\alpha(a):=a\circ\psi^0$ clearly defines a $*$-isomorphism $\alpha:C_0(E^0)\to C_0(F^0)$. For $x\in X(E)$ define $Ux:=x\circ\psi^1$.
For $v\in F^1$ we have
\begin{equation*}\sum_{e\in s_F^{-1}(v)}|Ux(e)|^2=\sum_{e\in s_F^{-1}(v)}|x(\psi^1(e))|^2=\sum_{f\in s_E^{-1}(\psi^0(v))}|x(f)|^2\end{equation*}
(using the fact that if $s_E(e)=v$ then $s_F(\psi^1(e))=\psi^0(v)$) and so $\widehat{Ux}(v)=\hat{x}(\psi^1(v))$. As $\hat{x}\in C_0(E^0)$ it follows immediately that $\widehat{Ux}\in C_0(F^0)$, i.e. $Ux\in X(F)$. Identical arguments show that $U^{-1}y:=y\circ(\psi^1)^{-1}$ is a map from $X(F)$ to $X(E)$ which is a two-sided inverse for $U$. Hence $U:X(E)\to X(F)$ is a bijection which is naturally linear.

Given $x\in X(E)$, $a\in C_0(E^0)$, and $e\in E^1$ we have
\begin{align*}
U(x\cdot a)&=(x(a\circ s_E))\circ\psi^1=(x\circ\psi^1)(a\circ s_E\circ\psi^1)=(Ux)(a\circ\psi^0\circ s_F)=Ux\cdot\alpha(a)\\
U(\phi_E(a)x)&=((a\circ r_E)x)\circ\psi^1=(a\circ r_E\circ\psi^1)(x\circ\psi^1)=(a\circ \psi^0\circ r_F)(Ux)=\phi_F(\alpha(a))Ux
\end{align*}
and, given $v\in F^0$,
\begin{align*}
\langle Ux,y\rangle(v)&=\sum_{e\in s_F^{-1}(v)}\overline{Ux(e)}y(e)=\sum_{e\in s_F^{-1}(v)}\!\!\overline{x(\psi^1(e))}y(e)=\!\!\!\sum_{f\in s_E^{-1}(\psi^0(v))}\hspace{-.5cm}\overline{x(f)}y((\psi^1)^{-1}(f))\\
&=\sum_{f\in s_E^{-1}(\psi^0(v))}\overline{x(f)}U^{-1}y(f)=\langle x,U^{-1}y\rangle(\psi^0(v))= \alpha(\langle x,U^{-1}y\rangle)(v)
\end{align*}
(the first inner product is that of $X(F)$ and the later two are that of $X(E)$). Thus the pair of $U$ and $\alpha$ satisfies the definition of a coherent unitary equivalence.
\end{proof}

Not every coherent unitary equivalence comes from a graph isomorphism in the sense of the preceding Proposition. As a simple example, consider the graph $E$ with but a single vertex and two edges. In this case $C_0(E^0)=\C$ and $X(E)=\C^2$. Hence any unitary $U\in M_2(\C)$ forms (with the identify on $C_0(E^0)$) a coherent unitary equivalence. However, the only such equivalences arising from graph isomorphisms would be those of the two permutation matrices.

\begin{proposition}\label{isotoeplitz} If there is a coherent unitary equivalence between $X$ and $Y$ then $\T_X$ and $\T_Y$ are $*$-isomorphic. 
\end{proposition}
\begin{proof} Let $A$ and $B$ be the coefficient $C^*$-algebras for $X$ and $Y$, respectively.
Suppose that $(U,\alpha)$ is a coherent unitary equivalence between $X$ and $Y$ and let $(\sigma,\pi)$ be a Toeplitz representation of $Y$. For $x\in X$ and $a\in A$
\begin{align*}
\sigma(U(x\cdot a))&=\sigma(Ux\alpha(a))=\sigma(Ux)\pi(\alpha(a))\\
\sigma(U(\phi_X(a)x))&=\sigma(\alpha(a)Ux)=\pi(\alpha(a))\sigma(Ux)
\end{align*}
and for $x_1,x_2\in X$
\begin{equation*}\pi\circ\alpha(\langle x_1,x_2\rangle_A)=\pi(\langle Ux_1,Ux_2\rangle_B)=\sigma(Ux_1)^*\sigma(Ux_2).\end{equation*}
Hence $(\sigma\circ U,\pi\circ \alpha)$ is a Toeplitz representation of $X$.

In particular, $(\sigma_Y\circ U,\pi_B\circ\alpha)$ is a Toeplitz representation of $X$ where $(\sigma_Y,\pi_B)$ is the universal Toeplitz representation of $Y$ in $\T_Y$. By the universal property of $\T_X$, there is a $*$-homomorphism $\theta:\T_X\to\T_Y$ such that $\theta\circ\sigma_X=\sigma_Y\circ U$ and $\theta\circ\pi_A=\pi_B\circ \alpha$, where $(\sigma_X,\pi_A)$ is the universal representation of $X$ in $\T_X$.

Similarly $(\sigma_X\circ U^{-1},\pi_A\circ \alpha^{-1})$ is a Toeplitz representation of $Y$ and induces a $*$-homomorphism $\theta':\T_Y\to \T_X$ for which $\theta'\circ\sigma_Y=\sigma_X\circ U^{-1}$ and $\theta'\circ\pi_B=\pi_A\circ \alpha^{-1}$. Thus
$$\sigma_Y=\sigma_Y\circ U\circ U^{-1}=\theta\circ\sigma_X\circ U^{-1}=\theta\circ\theta'\circ\sigma_Y$$
and similarly $\pi_B=\theta\circ\theta'\circ\pi_B$. Since the identity $id$ on $\T_Y$ also has the property that $\pi_B=id\circ\pi_B$ and $\sigma_Y=id\circ\sigma_Y$, it follows by the universal property of $\T_Y$ that $\theta\circ\theta'=id$. Identical reasoning verifies that $\theta'\circ\theta$ is the identity on $\T_X$. Thus $\theta$ is our desired $*$-isomorphism.
\end{proof}

Going forward we will be exclusively interested in Toeplitz algebras associated to graph correspondences, and so offer the following corollary.
\begin{corollary}\label{autosfromcoherent} Let $E$ and $F$ be graphs. If $(U,\alpha)$ is a coherent unitary equivalence between $X(E)$ and $X(F)$ then there is a $*$-isomorphism $\Gamma_{U,\alpha}:\T_{X(E)}\to\T_{X(F)}$ for which $\Gamma_{U,\alpha}(\delta_e)=U\delta_{e}$ and $\Gamma_{U,\alpha}(\delta_v)=\alpha(\delta_v)$ for all $e\in E^1$ and $v\in E^0$.
\end{corollary}
This is immediately seen from the proof of the previous Proposition if we recall that we identify $X(E)$ and $X(F)$ with their images in $\T_{X(E)}$ and $\T_{X(F)}$, respectively, under the appropriate universal maps. This also implies that if $E$ and $F$ are isomorphic graphs then $\T_{X(E)}$ and $\T_{X(F)}$ are $*$-isomorphic, which is unsurprising.

\section{ENDOMORPHISMS FROM GRAPHS}

Throughout this section we will let $E=(E^0,E^1,r,s)$ be a given graph whose vertex and edge sets are at most countable. All $*$-representations will be assumed non-degenerate.

\begin{proposition} Given a $*$-representation $\tau:\T_{X(E)}\to B(H)$, the assignments
\begin{equation*}Ad_\tau(w):=\sum_{e\in E^1}\tau(\delta_e)w\tau(\delta_e)^*\end{equation*}
(the sum is taken as a SOT limit) define a $*$-endomorphism $Ad_\tau$ of the von Neumann algebra $W=\{\tau(\delta_v):v\in E^0\}'$ (this notation will denote the relative commutant in $B(H)$).
\end{proposition}
\begin{proof} First, notice that for $e\in E^1$ and $w\in W$ the term $\tau(\delta_e)w\tau(\delta_e)^*$ has its support projection contained in $\tau(\delta_e^*\delta_e)$. Since the partial isometries $\tau(\delta_e)$ have mutually orthogonal ranges, it follows that for every $h\in H$, $\tau(\delta_e)w\tau(\delta_e)^*h$ is nonzero for at most one $e\in E^1$. Thus the sum converges in the SOT.

Certainly $Ad_\tau$ is linear and has $Ad_\tau(w^*)=Ad_\tau(w)^*$ for each $w\in W$. Given $w_1,w_2\in W$ we find that
\begin{align*}Ad_\tau(w_1)Ad_\tau(w_2)&=\bigg(\sum_{e\in E^1}\tau(\delta_e)w_1\tau(\delta_e)^*\bigg)\bigg(\sum_{f\in E^1}\tau(\delta_f)w_2\tau(\delta_f)^*\bigg)\\
	&= \sum_{e,f\in E^1}\tau(\delta_e)w_1\tau(\delta_e)^*\tau(\delta_f)w_2\tau(\delta_f)^*\\
	&= \sum_{e\in E^1}\tau(\delta_e)w_1\tau(\delta_{s(e)})w_2\tau(\delta_e)^*\\
	&=\sum_{e\in E^1}\tau(\delta_e)\tau(\delta_{s(e)})w_1w_2\tau(\delta_e)^*\\
	&=\sum_{e\in E^1}\tau(\delta_e)w_1w_2\tau(\delta_e)^*\\
	&=Ad_\tau(w_1w_2)
	\end{align*}
and so $Ad_\tau$ is multiplicative. Note that any potential issues with SOT-convergence of the product are circumvented by $E^1$ being at most countable. All that remains is to verify that $Ad_\tau(w)\in W$ for each $w\in W$. To that end we first note that $\delta_e^*\delta_v=\delta_e^*$ if $v=r(e)$ and is zero otherwise. By taking adjoints, $\delta_v\delta_e=\delta_e$ if $v=r(e)$ and is zero otherwise. Thus, given $w\in W$ and $v\in E^0$ we find
\begin{equation*}Ad_\tau(w)\tau(\delta_v)=\sum_{e\in r^{-1}(v)}\tau(\delta_e)w\tau(\delta_e)^*=\tau(\delta_v)Ad_\tau(w)\end{equation*}
and so $Ad_\tau(w)$ commutes with each $\tau(\delta_v)$.
\end{proof}

We note that the von Neumann algebra $W=\{\tau(\delta_v):v\in E^0\}$, because the $\tau(\delta_v)$ are a family of mutually orthogonal projections, is precisely equal to $\displaystyle{\bigoplus_{v\in E^0} \tau(\delta_v)B(H)\tau(\delta_v)}$ which is a sum of Type I factors. 

The following is a construction which we believe to be folklore, but our use of it is motivated by observations made by Muhly and Solel \cite{muhlysolel}. Given a $*$-representation $\tau:\T_{X(E)}\to B(H)$ let $W=\{\tau(\delta_v):v\in E^0\}'$. The space
\begin{equation*}\mathcal{I}_\tau:=\left\{T\in B(H):Ad_\tau(w)T=Tw\ ,\ w\in W \right\}\end{equation*}
is a $C^*$-correspondence over $W'$. The left and right actions of $W'$ are simply multiplication within $B(H)$ and the $W'$-valued inner product is defined by $\langle T,S\rangle_{W'}:=T^*S$.

Because our endomorphism is of the form $Ad_\tau$,  we can say more: for $w\in W$ and $e\in E^1$
\begin{equation*}Ad_\tau(w)\tau(\delta_e)=\sum_{f\in E^1}\!\!\tau(\delta_f)w\tau(\delta_f)^*\tau(\delta_e)=\tau(\delta_e)w\tau(\delta_{s(e)})=\tau(\delta_e)\tau(\delta_{s(e)})w=\tau(\delta_e)w\end{equation*}
and so $\tau(\delta_e)\in\mathcal{I}_\tau$ for each $e\in E^1$. As $\tau(\delta_v)\in W'$ for each $v\in E^0$ we finally have $\tau(X(E))\subseteq \mathcal{I}_\tau$.

\begin{theorem}\label{theorem1} Suppose that $E$ and $F$ are graphs and $\tau_1:\T_{X(E)}\to B(H)$ and $\tau_2:\T_{X(F)}\to B(H)$ are two faithful $*$-representations. If $Ad_{\tau_1}=Ad_{\tau_2}$ on $W=\{\tau_1(\delta_v):v\in E^0\}'=\{\tau_2(\delta_v):v\in F^0\}'$ then there is a coherent unitary equivalence $(U,\alpha)$ between $X(E)$ and $X(F)$ such that $\tau_2=\tau_1\circ\Gamma_{U,\alpha}$. 
\end{theorem}
Here $\Gamma_{U,\alpha}$ is the $*$-isomorphism from $\T_{X(E)}$ to $\T_{X(F)}$ arising from $(U,\alpha)$ as given in Corollary \ref{autosfromcoherent}.
\begin{proof} 
Since $\{\tau_1(\delta_v):v\in E^0\}$ and $\{\tau_2(\delta_v):v\in F^0\}$ are sets of orthogonal projections with the same commutant they are in fact equal, and in particular $E^0$ and $F^0$ are of the same cardinality. To ease notation we'll denote these projections by $P_v$, $v\in E^0$, (with no assumption that $P_v=\tau_1(\delta_v)$ or similar) hence
\begin{equation*}\{P_v:v\in E^0\}=\{\tau_1(\delta_v):v\in E^0\}=\{\tau_2(\delta_v):v\in F^0\}.\end{equation*}

As $Ad_{\tau_1}=Ad_{\tau_2}$ we have that $\mathcal{I}_{\tau_1}=\mathcal{I}_{\tau_2}$ and we'll call this module simply $\mathcal{I}$.

As $\tau_1(\delta_e)\in\I$ for each $e\in E^1$ we have
\begin{equation*}\tau_1(\delta_e)=\tau_1(\delta_e)I=Ad_{\tau_2}(I)\tau_1(\delta_e)=\sum_{f\in F^1}\tau_2(\delta_f)\tau_2(\delta_f)^*\tau_1(\delta_e)\end{equation*}
hence $\tau_1(\delta_e)$ is in the $W'$-submodule of $\I$ generated by $\tau_2(X(E))$. Similarly, for each $f\in F^1$, $\tau_2(\delta_f)$ is in the $W'$-submodule generated by $\tau_1(X(E))$. Thus $\tau_1(X(E))$ and $\tau_2(X(E))$ generate the same $W'$-submodule of $\I$.

Given $e\in E^1$ and $f\in F^1$ we have seen that 
\begin{equation*}\tau_2(\delta_f)^*\tau_1(\delta_e)\in W'=\{P_v:v\in E^0\}''=\ell^\infty(\{P_v:v\in E^0\}).\end{equation*}
Notice however that $\tau_2(\delta_f)^*\tau_1(\delta_e)\tau_1(\delta_v)=0$ unless $v=s(e)$ and hence $\tau_2(\delta_f)^*\tau_1(\delta_e)$ is a multiple of $\tau_1(\delta_{s(e)})$ only, i.e. is an element of $C_0(\{P_v:v\in E^0\})$. Since before we obtained $\tau_1(\delta_e)=\sum_{f\in F^1}\tau_2(\delta_f)\tau_2(\delta_f)^*\tau_1(\delta_e)$ for all $e\in E^1$, it now follows that $\tau_1(X(E))$ and $\tau_2(X(E))$ generate the same correspondence over $C_0(\{P_v:v\in E^0\})$. It is important to note that this correspondence has three different actions of $C_0(\{P_v:v\in E^0\})$: the ones inherited through $\tau_1$ and $\tau_2$ and simple operator multiplication in $B(H)$.
 
Finally we have that $\tau_1(C_0(E^0))=\tau_2(C_0(F^0))$ and $\tau_1(X(E))=\tau_2(X(F))$ as sets and so, because both representations are faithful by hypothesis, $\tau_2^{-1}\circ\tau_1$ is a well-defined bijection between $X(E)$ and $X(F)$ and between $C_0(E^0)$ and $C_0(F^0)$. Denote by $U$ and $\alpha$ the restrictions of $\tau_2^{-1}\circ\tau_1$ to $X(E)$ and to $C_0(E^0)$, respectively.

Given $x\in X(E)$, $y\in X(F)$, and $a\in C_0(E^0)$ we have
\begin{align*}U(xa)&=\tau_2^{-1}\circ\tau_1(xa)=\tau_2^{-1}\circ\tau_1(x)\tau_2^{-1}\circ\tau_1(a)=(Ux)\alpha(a),\\
U(\phi(a)x)&=\tau_2^{-1}\circ\tau_1(\phi(a)x)=\tau_2^{-1}\circ\tau_1(a)\tau_2^{-1}\circ\tau_1(x)=\alpha(a)Ux,\\
\langle Ux,y\rangle&=[\tau_2^{-1}\circ\tau_1(x)]^*y=\tau_2^{-1}\circ\tau_1(x^*\tau_1^{-1}\circ\tau_2(y))=\alpha(\langle x,\tau_1^{-1}\circ\tau_2(y)\rangle)=\alpha(\langle x,U^{-1}y\rangle).\end{align*}
and so $(U,\alpha)$ is a coherent unitary equivalence between $X(E)$ and itself.

It follows from Corollary \ref{autosfromcoherent} that $(U,\alpha)$ induces a $*$-isomorphism $\Gamma_{U,\alpha}:\T_{X(E)}\to\T_{X(F)}$ and, by construction, $\tau_2\circ \Gamma_{U,\alpha}=\tau_1$.
\end{proof}

Notice we have shown that if $\tau_1:\T_{X(E)}\to B(H)$ and $\tau_2:\T_{X(F)}\to B(H)$ generate identical $*$-endomorphims $Ad_{\tau_1}$ and $Ad_{\tau_2}$ then $\T_{X(E)}$ and $\T_{X(F)}$ are $*$-isomorphic as $C^*$-algebras. Hence in subsequent results we will concern ourselves only with two representations $\tau_1$, $\tau_2$ of the same Toeplitz algebra $\T_{X(E)}$.

Our result is a generalization of Laca's \cite[Proposition 2.2]{laca}. When $E$ is the graph with a single vertex and $n\in\N\cup\{ \infty\}$ edges we have already seen that $\T_{X(E)}=\E_n$. If $\tau_1$ and $\tau_2$ are faithful and nondegenerate then $W=B(H)$. The map $\alpha$ is the identity on $C_0(E^0)=\C$ and $U$ is a unitary operator on the Hilbert space $X(E)=\ell^2(\{v_1,...,v_n\})$. Hence $\Gamma_{U,\alpha}$ is a $*$-automorphism of $\E_n$ which fixes the Hilbert space $X(E)$ and implements the equivalence between the two representations.

We will conclude this section with a discussion of conjugacy conditions for endomorphisms of the type we've been examining. Recall that two endomorphisms $\alpha$ and $\beta$ are said to be \emph{conjugate} if there is an automorphism $\gamma$ such that $\alpha\circ\gamma=\gamma\circ\beta$.

\begin{lemma}\label{spatialdiagonal} If $P_1,P_2,...\in B(H)$ is an at most countable family of orthogonal projections and $\gamma$ is a $*$-automorphism of $W=\{P_1,P_2,...\}'$ then there exists a unitary $U\in B(H)$ such that $\gamma(w)=UwU^*$ for all $w\in W$.
\end{lemma}
\begin{proof} Note that for each $n$, $\gamma$ restricts to a $*$-isomorphism $\gamma_n$ between $P_nB(H)P_n=B(P_nH)$ and $\gamma(P_n)B(H)\gamma(P_n)=B(\gamma(P_n)H)$. Such isomorphisms are always spatial and so there are unitaries $U_n:B(P_nH)\to B(\gamma(P_n)H)$ such that $\gamma_n(w)=U_nwU_n^*$. It is then immediate that $U:=U_1\oplus U_2\oplus...$ is a unitary in $B(H)$ and $UwU^*=\gamma(w)$ for each $w\in W$.
\end{proof}

\begin{theorem}\label{theorem2} Suppose that $\tau_1,\tau_2:\T_{X(E)}\to B(H)$ are two faithful $*$-representations such that $Ad_{\tau_1}$ and $Ad_{\tau_2}$ are conjugate $*$-endomorphisms of $W=\{\tau_1(\delta_v):v\in E^0\}'=\{\tau_2(\delta_v):v\in E^0\}'$. Then there is a coherent unitary equivalence $(U,\alpha)$ between $X(E)$ and itself such that $\tau_2$ and $\tau_1\circ\Gamma_{U,\alpha}$ are unitarily equivalent $*$-representations.
\end{theorem}

\begin{proof} Let $\gamma$ be an $*$-automorphism of $W$ such that $Ad_{\tau_1}\circ\gamma=\gamma\circ Ad_{\tau_2}$ and let $V\in B(H)$ be the unitary for which $\gamma(w)=VwV^*$ according the Lemma \ref{spatialdiagonal}. Then $Ad_{\tau_2}(w)=V^*Ad_{\tau_1}(VwV^*)V$ for all $w\in W$. Define $\kappa:\T_{X(E)}\to B(H)$ by $\kappa(t):=V\tau_1(t)V^*$ and note that $\kappa$ is a $*$-representation of $\T_{X(E)}$ such that 
\begin{equation*}Ad_{\kappa}(w)=\sum_{e\in E^1}\kappa(\delta_e)w\kappa(\delta_e)^*=\sum_{e\in E^1}V\tau_1(\delta_e)V^*wV\tau_1(\delta_e)^*V^*=VAd_{\tau_1}(V^*wV)V^*\end{equation*}
and so $Ad_{\kappa}=Ad_{\tau_2}$ on $W$. Applying Theorem \ref{theorem1} we obtain a coherent unitary equivalence $(U,\alpha)$ inducing the $*$-automorphism $\Gamma_{U,\alpha}$ of $\T_{X(E)}$ such that $\tau_2=\kappa\circ\Gamma_{U,\alpha}$. As now $\tau_2(t)=V[\tau_1\circ\Gamma_{U,\alpha}(t)]V^*$ for each $t\in \T_{X(E)}$, we have that $\tau_2$ and $\tau_1\circ\Gamma_{U,\alpha}$ are unitarily equivalent, as desired.
\end{proof}

\section{GRAPHS FROM ENDOMORPHISMS}

In this final section we will demonstrate that all $*$-endomorphisms of von Neumann algebras which are sums of Type I factors are obtained in the natural way from representations of Toeplitz algebras for graph correspondences. Our result is a significant generalization of \cite[Theorem 3.9]{brenken2} which places technical restrictions on the endomorphisms. In the case of unital endomorphisms our results are comparable.

\begin{theorem}\label{theorem3} Let $W=\bigoplus W_i\subseteq B(H)$ be a countable sum of Type I factors. Let $P_1,P_2,...\in B(H)$ be projections such that $W_i=P_iB(H)P_i$. If $\alpha$ is a normal $*$-endomorphism of $W$ then there exists a graph $E$ and $*$-representation $\tau:\T_{X(E)}\to B(H)$ such that $\alpha=Ad_\tau$.
\end{theorem}
\begin{proof} Without loss of generality we may assume that $\sum P_i=I$. If this were not the case we may define $P_0=(\sum P_i)^\perp$ and $W_0=P_0B(H)P_0$ and extend $\alpha$ to a normal $*$-endomorphism of $W\oplus W_0$ with $\alpha|_{W_0}=id$. 

For $i>0$ define $H_i=P_iH$. For $i,j>0$ and $x\in W$ define $\alpha_{ij}(x)=P_j\alpha(P_ix)$. Then $\alpha_{ij}$ restricts to a $*$-homomorphism between $B(H_i)=P_iB(H)P_i=W_i$ and $B(H_j)=P_jB(H)P_j=W_j$ as seen by
\begin{equation*}P_j\alpha(P_ix)P_j\alpha(P_iy)=P_j\left(P_j\alpha(P_ix)\right)\alpha(P_iy)=P_j\alpha(P_ixP_iy)=P_j\alpha(P_ixy).\end{equation*}

So $\alpha_{ij}$ is a $*$-homomorphism between two Type I factors, and thus by \cite[Proposition 2.1]{arveson} if $\alpha_{ij}$ is nonzero there exists $n_{ij}\in\N\cup\{\infty\}$ and isometries $V_k^{(ij)}\in B(H_i,H_j)$, $k=1,...,n_{ij}$ such that $\alpha_{ij}|_{B(H_i)}(T)=\sum_{k=1}^{n_{ij}}V_k^{(ij)}TV_k^{(ij)*}$.
We will identify the $V_k^{(ij)}$ with their associated partial isometries in $B(H)$, so that $V_k^{(ij)*}V_k^{(ij)}=P_i$ and $V_k^{(ij)}V_k^{(ij)*}\leq P_j$.

Set $E^0:=\{P_1,P_2...,\}$ and $E^1:=\bigcup_{i,j}\{V_k^{(ij)}:k=1,...,n_{ij}\}$. Define maps $r,s:E^1\to E^0$ by $r(V_k^{(ij)})=P_j$ and $s(V_k^{(ij)})=P_i$. Then $E=(E^0,E^1,r,s)$ is a graph. It is trivial to see that the identity maps on $E^1$ and $E^0$ extend to a Toeplitz representation of $X(E)$ which in turn induces a $*$-representation $\tau:\T_{X(E)}\to B(H)$.

Finally, we have that for each $x\in W$
\begin{equation*}\alpha(x)=\sum_{i,j>0} P_j\alpha(P_ix)=\sum_{i,j>0}\alpha_{ij}(x)=\sum_{i,j>0}\sum_{k=1}^{n_{ij}} V_k^{(ij)}xV_k^{(ij)*}=\sum_{f\in E^1}\tau(\delta_f)x\tau(\delta_f)^*\end{equation*}
as desired.
\end{proof}
As a consequence, the possible $*$-endomorphisms of a given von Neumann algebra $W=\{P_1,P_2,...\}'$ coincide, up to conjugacy, with the coherent unitary equivalence classes of $C^*$-correspondences over $C_0(\{P_1,P_2,...\})$. 

\subsection{Cuntz-Pimsner Algebras} To discuss the special case of unital $*$-endomorphisms we will need to briefly sketch the defining features and universal properties of the so-called Cuntz-Pimsner algebras. The relationship of these Cuntz-Pimnser algebras to our Toeplitz algebras is analogous to the relationship between the classical Toeplitz algebras $\E_n$ and classical Cuntz algebras $\O_n$. The Cuntz-Pimsner algebras were originally defined in \cite{pimsner} though our treatment will take its cues from \cite{fowlerraeburnmuhly} and \cite{fowlerraeburn}.

Recalling that $X(E)$ is a $C^*$-correspondence over $C_0(E^0)$, notice that $I=C_0(\{v\in\ E^0:r^{-1}(v)\text{ is finite }\})$ is a closed, two-sided ideal in $C_0(E^0)$. As noted in \cite[Proposition 4.4]{fowlerraeburn} (and recalling our modern reversal of $r$ and $s$ from the presentation in that work), elements of $I$ are precisely those elements of $C_0(E^0)$ whose left action on $X(E)$ is compact. Since our graphs have at most countable vertices and edges, $X(E)$ is countably generated as a $C^*$-correspondence over $C_0(E^0)$. It follows, \cite[Remark 3.9 following Definition 3.8]{pimsner} that the \emph{Cuntz-Pimnser algebra} of $X(E)$ is the $C^*$-algebra $\O_{X(E)}$ which is universal for Toeplitz representations $(\sigma,\pi)$ (as in Definition \ref{toeplitzrepdef}) which additionally satisfy:
\begin{enumerate}
\item[(iv)] $\displaystyle{\pi(a)=\sum_{f\in E^1}\sigma(\phi(a)\delta_f)\sigma(\delta_f)^*}$ for all $a\in I$.
\end{enumerate}
or, equivalently  (see \cite[Example 1.5]{fowlerraeburnmuhly}),
\begin{enumerate}
\item[(iv)$'$] $\displaystyle{\pi(\delta_v)=\sum_{f\in r^{-1}(v)}\sigma(\delta_f)\sigma(\delta_f)^*}$ for all $v\in E^0$ with $|r^{-1}(v)|<\infty$.
\end{enumerate}
Representations satisfying (i)-(iv)$'$ are often termed ``coisometric" Toeplitz representations. With this characterization of $\O_{X(E)}$ we may recognize that Theorem \ref{theorem3} has more to say when the endomorphism is unital.

\begin{cor} Let $W=\bigoplus W_i\subseteq B(H)$ be a countable sum of Type I factors. Let $P_1,P_2,...\in B(H)$ be projections such that $W_i=P_iB(H)P_i$. If $\alpha$ is a normal, unital $*$-endomorphism of $W$ then there exists a graph $E$ and $*$-representation $\tau:\O_{X(E)}\to B(H)$ such that $\alpha=Ad_\tau$.
\end{cor}
\begin{proof} We define the partial isometries $V_k^{(ij)}$ and graph $E$ in the same manner as in the proof of Theorem \ref{theorem3}. Since $\alpha$ is unital we have

$$P_n=P_n\alpha(I)=\sum_{i,j>0}\sum_{k=1}^{n_{ij}} P_nV_k^{(ij)}V_k^{(ij)*}=\sum_{i>0}\sum_{k=1}^{n_{in}} V_k^{(in)}V_k^{(in)*}$$
for every $P_n$. Recalling that the representation of $(\sigma,\pi)$ $E$ on $B(H)$ is the pair of identity maps, this becomes
$$\pi(\delta_v)=\sum_{f\in E^1}\pi(\delta_v)\sigma(\delta_f)\sigma(\delta_f)^*=\sum_{f\in r^{-1}(v)}\sigma(\delta_f)\sigma(\delta_f)^*$$
for all $v\in E^0$. In particular this holds for all edges $v$ with $r^{-1}(v)$ finite (i.e.\ all $P_j$ such that $\{V_k^{(ij)}: i>0, k=1,...,n_{ij}\}$ is finite). Thus condition (iv)$'$ is satisfied and the identity maps on $E^0$ and $E^1$ induce a representation of $\O_{X(E)}$. The fact that $\alpha=Ad_\tau$ follows as before.
\end{proof}

\bibliographystyle{plain}
\bibliography{gipsonbibl}

\begin{thebibliography}{1}

\bibitem{arveson}
William Arveson.
\newblock Continuous analogues of {F}ock space.
\newblock {\em Mem. Amer. Math. Soc.}, 80(409):iv+66, 1989.

\bibitem{brenken1}
Berndt Brenken.
\newblock Cuntz-{K}rieger algebras and endomorphisms of finite direct sums of
  type {${\rm I}_\infty$} factors.
\newblock {\em Trans. Amer. Math. Soc.}, 353(10):3835--3873, 2001.

\bibitem{brenken2}
Berndt Brenken.
\newblock Endomorphisms of type {I} von {N}eumann algebras with discrete
  center.
\newblock {\em J. Operator Theory}, 51(1):19--34, 2004.

\bibitem{fowlerraeburnmuhly}
Neal~J. Fowler, Paul~S. Muhly, and Iain Raeburn.
\newblock Representations of {C}untz-{P}imsner algebras.
\newblock {\em Indiana Univ. Math. J.}, 52(3):569--605, 2003.

\bibitem{fowlerraeburn}
Neal~J. Fowler and Iain Raeburn.
\newblock The {T}oeplitz algebra of a {H}ilbert bimodule.
\newblock {\em Indiana Univ. Math. J.}, 48(1):155--181, 1999.

\bibitem{laca}
M.~Laca.
\newblock Endomorphisms of {$B(H)$} and {C}untz algebras.
\newblock {\em J. Operator Theory}, 30(1):85--108, 1993.

\bibitem{longo}
Roberto Longo.
\newblock Simple injective subfactors.
\newblock {\em Adv. in Math.}, 63(2):152--171, 1987.

\bibitem{muhlysolel}
Paul~S. Muhly and Baruch Solel.
\newblock Quantum {M}arkov processes (correspondences and dilations).
\newblock {\em Internat. J. Math.}, 13(8):863--906, 2002.

\bibitem{pimsner}
Michael~V. Pimsner.
\newblock A class of {$C^*$}-algebras generalizing both {C}untz-{K}rieger
  algebras and crossed products by {${\bf Z}$}.
\newblock In {\em Free probability theory ({W}aterloo, {ON}, 1995)}, volume~12
  of {\em Fields Inst. Commun.}, pages 189--212. Amer. Math. Soc., Providence,
  RI, 1997.

\end{thebibliography}
\end{document}